\documentclass[12pt]{amsart}

\usepackage{cite}
\usepackage{amsmath}
\usepackage{amstext}
\usepackage{amssymb}

\usepackage{esint}

\usepackage{amsthm}

\usepackage{bm}

\theoremstyle{plain}
\newtheorem{theorem}{Theorem}[section]
\newtheorem{lemma}[theorem]{Lemma}

\theoremstyle{definition}
\newtheorem{definition}[theorem]{Definition}
\newtheorem{example}[theorem]{Example}
\numberwithin{equation}{section}

\DeclareMathOperator{\Rp}{Re}

\DeclareMathOperator*{\esssup}{ess\,sup}
\DeclareMathOperator*{\essinf}{ess\,inf}

\DeclareMathOperator{\rc}{rc}
\DeclareMathOperator{\lv}{lv}

\newcommand{\family}[1]{#1}

\usepackage{cite}

\begin{document}

\title[Good Lambda and Variable Orlicz Hardy Spaces]%
{Good Lambda Inequalities and Variable Orlicz Hardy Spaces}
\author{Timothy Ferguson}
\address{Department of Mathematics\\University of Alabama\\Tuscaloosa, AL}
\email{tjferguson1@ua.edu}

\date{\today}

\newcommand{\sci}{slowly changing on intervals}
\newcommand{\scifm}{\sci\ for modulus}
\newcommand{\udoub}{uniformly doubling}

\begin{abstract}
  We prove a general theorem showing that
  local good-$\lambda$ inequalities imply bounds
  in certain variable Orlicz spaces.  We use this
  to prove results about 
  variable Orlicz Hardy spaces 
  in the unit disc.
\end{abstract}

\dedicatory{In memoriam Peter Duren}





\maketitle

\section{Introduction}

There are two main topics in this paper.  
To understand the first topic, it helps to know that results about
variable exponent Lebesgue spaces are often proven using
Rubio de Francia's extrapolation theory
\cite{DCU_VarLebesgueBook}.  This relies on
establishing weighted $L^p$ bounds, which are in turn
often proven by using local good-$\lambda$ inequalities.  The
first part of this paper shows how to obtain bounds for 
certain variable exponent $L^p$ spaces
directly from
local good-$\lambda$ inequalities.  Our result also holds for
variable Orlicz spaces. 

The second topic of the paper is the application of the results
mentioned above to develop a theory of variable Orlicz
Hardy spaces.
In particular, we give examples of variable
exponent Hardy spaces in which the exponent can approach $0$ in a sense.    
%
%
We work on the unit disc, since this seems to be the
case easiest to understand at first.
For work on variable exponent Hardy spaces, see
\cite{Kokilashvili-Paatashvili2006, Kokilashvili-Paatashvili2008,
  Kokilashvili-Paatashvili2015, DCU-VariableHardy, MR3220151, MR2899976}.

We now introduce some preliminaries needed to understand the
paper.  

\begin{definition}
For each
$x$ in the unit circle 
suppose there is a continuous increasing function
$\Phi_x:[0,\infty) \rightarrow [0, \infty)$.
We will also require that for each
$x$, both $\Phi_x(0) = 0$ and $\Phi_x(t) \rightarrow \infty$ as
$t \rightarrow \infty$. 
A family of functions $\family{\Phi}$ as above will be called a 
{\itshape modular family}. 
\end{definition}
We sometimes use the term ``family'' for short. 

We will often work in the case where for each
interval $I$ there are increasing functions
$\Phi_{I,+}$ and $\Phi_{I,-}$ such that
\[
  \Phi_{I,-}(t) \leq \Phi_x(t) \leq \Phi_{I,+}(t)
\]
for all $x$ in the closure of $I$
and for all $ t \in [0,\infty)$.  (By interval we mean an
arc of the unit circle in this context).  Often we take
$\Phi_{I,+}(t) = \sup_{x \in \overline{I}} \Phi_x(t)$ and
similarly for $\Phi_{I, -}$ with the supremum replaced by
an infimum.  We will assume that 
the $\Phi_{I, \pm}$ are the same for any interval and its closure,
and
that if $I \subset J$ that
\[
  \Phi_{J, -}(t) \leq   \Phi_{I, -}(t)  \leq   \Phi_{I, +}(t)  \leq
  \Phi_{J, +}(t)
\]
for all $t \in [0,\infty)$.

For any function $\Phi_x$ in our
family, 
we will define $\Phi_x(z) = \Phi_x(|z|)$ if
$z$ is not a nonnegative real number.

\begin{definition}
If there is some constant
$C'$ such that $\Phi_x(2\lambda) \leq C' \Phi_x(\lambda)$
for any $x$, then the 
modular family $\family{\Phi}$ of these functions is said to be \udoub.  
$\mathfrak{\Phi}$
\end{definition}

For a function $f$ we define its
modular by
\[
  \rho_{\Phi}(f) =
  \int_0^1 \Phi_x(f(x)) \, dx.
\]
We define its norm by
\[
  \|f\|_{\Phi} = \inf\left\{\lambda:
    \int_0^1 \Phi_x(f(x)/\lambda) \, dx < 1
    \right\}.
\]
We sometimes omit the subscript on the norm.  We let $L_{\Phi}$
be the set of all functions $f$ such that $\|f\|_{\Phi}$ is finite.
Note that the name of ``norm'' does not imply that 
$\| \cdot \|$ is a vector space norm.  

The dominated convergence theorem shows that if
$\rho_{\Phi}(f) < \infty$ then $\|f\|_{\Phi} < \infty$.  
However, there is in general no result that says if
$\rho_{\Phi}(f) < A$ then $\|f\|_{\Phi} < B$ for any
$A > 1$.  To see this, let
$\Phi_x(\lambda) = \Phi(\lambda) = (\log^+(\lambda))^2$
(this example will be relevant later).  Let
$f(x) = \chi_{0,1/n^2} e^{n\sqrt{A}}$ for large $n$.  Then
$\rho(f) = A$ but $\|f\| = e^{(\sqrt{A}-1)n}$.
If the family $\family{\Phi}$ is \udoub\
with constant $C'$ then we can say that
if $\|f\| \leq A$ then there is some number $B$ depending
only on $C'$ and $A$ such that $\rho(f) \leq B$.
Thus, for uniformly doubling families,
$L_{\Phi}$ is equivalent to the space of functions with
bounded modular. 

Because of the above observations, it is important to pay
attention to the distinction between the norm and the modular.  
Also because of the above observations, we introduce a
modified norm for $a > 0$ defined by 
\[
  \|f\|_{\Phi, (a)} = \inf\left\{\lambda:
    \int_0^1 \Phi_x(f(x)/\lambda) \, dx < a
  \right\}.
\]
We sometimes write $\|f\|_{(a)}$ if the family
$\family{\Phi}$ is clear. 
By similar reasoning to what we have written above,
we have that $\|f\|_{(a)}$ is finite if the
modular is finite, and if the family is \udoub\ and 
$\|f\|_{(a)}$ is finite then the modular is finite.  
However, there is no relation that
says that if $\rho(f) < A$ then $\|f\|_{(a)} < B$ for any
$A > a$.  Also, it is clear that
$\|f\|_{(a)} < \|f\|_{(b)}$ if $a > b$.
Something else to note is that if $\Psi(y) = a \Phi(y)$
then $\|\cdot\|_{\Psi,(a)} = \|\cdot\|_{\Phi}$. 

By the dominated convergence theorem, since $\Phi_x$ is continuous
for each $x$, if 
$\rho(f) < \infty$ then $\rho(c f)$ is a continuous function of 
$c$ for $|c| \leq 1$.  Also, if $\rho(f) < \infty$ and the family 
$\Phi$ is uniformly doubling, then $\rho(cf) < \infty$ for any 
$c$, so in this case $\rho(cf)$ is a continuous function of $c$ for 
all $c$.
If this is the case, then because for each $x$, $\Phi_x(0) = 0$ and
$\Phi_x(t) \rightarrow \infty$ as $t \rightarrow \infty$, then 
for each $a > 0$,
there will be a $c > 0$ such that $\rho(f/c) = a$.
This $c$ will be unique if each $\Phi_x$ is strictly increasing. 

We now introduce the definition of \sci.  We will
first motivate the definition.  
Suppose that there is some constant $C$ depending on
$B > 0$ such that if
$\rho_{\Phi}(f) \leq B$ then 
  on each interval in the level set $\{x: f(x) > \lambda\}$
  we have
  $\Phi_{I,+}(\lambda) \leq C \Phi_{I,-}(\lambda)$.
  Note that if $I$ is an interval in the level set of such
  an $f$, we must have $m(I) \leq B/\Phi_{I,-}(\lambda)$
  by the restriction on $f$.  So we may satisfy the condition
  $\Phi_{I,+}(\lambda) \leq C\Phi_{I,-}(\lambda)$ 
  if for each $\lambda > 0$
  and each interval $I$ of
  length at most $B/\Phi_{I,-}(\lambda)$, we have
  $\Phi_{I,+}(\lambda) \leq C \Phi_{I,-}(\lambda)$.
  Another way to say this is as follows.

\begin{definition}
  We say the modular family $\family{\Phi}$ is \sci\ if for every $B>0$ there
  is a constant $C > 0$ such that 
  for each interval $I$ and each number $b \geq m(I)/B$ we have
  \[
    \Phi_{I,+}(\Phi_{I,-}^{-1}(1/b)) \leq C/b.
  \]
  Equivalently, we may require that for each $B > 0$ there is
  a constant $C > 0$ such that for each number
  $w \leq \Phi^{-1}_{I,-}(B/m(I))$, we have
  \[
    \Phi_{I,+}(w) \leq C \Phi_{I,-}(w).
  \]
\end{definition}

We record our observations in a lemma.
\begin{lemma} \label{lemma:sci}
  Suppose that the modular family 
  $\family{\Phi}$ is \sci.  Then for any $B > 0$
  there is a constant $C$ such that if
  $\rho(f) \leq B$
  then on each interval $I$ in
  the level set of $\{x: f(x) > \lambda\}$ we have
  $\Phi_{I, +}(\lambda) \leq C \Phi_{I, -}(\lambda)$.
\end{lemma}

\begin{example}\label{ex:powersci}
 Suppose $\Phi_x(\lambda) = \lambda^{p(x)}$ for some log-H\"{o}lder 
 continuous exponent $p$ with $0 < p_{-} \leq p_{+} < \infty$,
 where $p_{-} = \essinf p(x)$ and $p_{+} = \esssup p(x)$.  
 Log-H\"{o}lder continuous means that
 \[
   |p(x) - p(y)| \leq \frac{C}{\log \frac{4\pi}{|x-y|}}
 \]
 for some constant $C$. 
Define $\Phi_{I,+}(\lambda) = \sup_{x \in \overline{I}} \Phi_x(\lambda)$ 
and similarly for $\Phi_{I,-}$ with $\inf$ in place of $\sup$.  

The family of such functions $\Phi$ is \sci. 
To see this, suppose that $I$ is an interval and that 
$x, y \in I$.  
Suppose $\Phi_{I,-}(w) \leq B/m(I).$  
Now if $w \geq 1$ this implies that 
$w \leq (B/m(I))^{1/p_{-}}$, and if 
$w \leq 1$ this implies that 
$w \leq (B/m(I))^{1/p_{+}}$.  Then 
$w \leq (B^{1/p_{+}})/|x-y|^{1/p_{+}}$ or 
$w \leq B^{1/p_{-}}/|x-y|^{1/p_{-}}$.    
Note that
\[
|x-y|^{p_{I,+}-p_{I,-}} \leq \exp\left(-C\log(|x-y|)/\log\left
    (\tfrac{4\pi}{|x-y|}\right)\right) 
\]
which is bounded by some constant independent of $I$.  Similarly,
$|x-y|^{p_{I,-}-p_{I,+}}$ is bounded by a constant independent of $I$.
Thus, $w^{p_{I,-}-p_{I,+}}$ and $w^{p_{I,+}-p_{I,-}}$ are both bounded
by a constant independent of $I$, so the family $\family{\Phi}$ is \sci.

We note also that if we define
$\Phi_x(\lambda) = \max(1, \lambda^{p(x)})$ that the family
$\family{\Phi}$ is still \sci.  The reasoning is similar to the 
above but slightly easier since then $\Phi_x(w) = 1$ for 
all $x$ and $w \leq 1$.  We might also define
\[
  \Phi_x(\lambda) = \begin{cases} \lambda &\text{ if $\lambda < 1$}\\
    \lambda^p(x) &\text{ if $\lambda \geq 1$}
  \end{cases}
\]
and the family will still be slowly chaning on intervals. 
\end{example}

\begin{example}\label{ex:logsci} 
Let $1 < s < \infty$ be given and suppose that 
$p(x)$ is $1/s$ H\"{o}lder continuous, where 
$0 \leq p(x) < \infty$ for all $x$. 
Let
\[
\Phi_{x}(t) = \begin{cases} 
            t &\text{ if $0 \leq t \leq 1$}\\
            \max(t^{p(x)}, (1+\log^+ t)^s) &\text{ if $t > 1$}.\\
            \end{cases}
\]
We claim the family defined by $\Phi_x$ is \sci.
Note that if $I$ is an interval
and $w \leq \Phi_{I,-}^{-1}(B/m(I))$ then 
$\Phi_{I,-}(w) \leq B/m(I)$ and so 
$(1+ \log^+(w))^s \leq B/m(I)$ if $w \geq 1$.
This implies that
$w \leq \exp(B^{1/s}/m(I)^{1/s})$ if $w \geq 1$. 

Now let $x,y \in I$ and assume as above that
 $w \leq \Phi_{I,-}^{-1}(B/m(I))$.
Assume without loss of generality that 
$p(x) > p(y)$.  If $w \leq 1$ then $\Phi_{x}(w) = \Phi_{y}(w)$. 
This also happens if $w^{p(x)} \leq (1+\log^+ w)^s$.  
If $w^{p(x)} \geq (1+\log^+ t)^s \geq w^{p(y)}$ then 
$\Phi_{x}(w)/\Phi_{y}(w) \leq w^{(p(x)-p(y))}$.  Thus in any 
event for $w \geq 1$ we have
\[
  \frac{\Phi_{x}(w)}{\Phi_{y}(w)} \leq
  w^{(p(x)-p(y))} \leq
  \exp\left(\frac{B^{1/s}|p(x)-p(y)|}{|x-y|^{1/s}}\right)
  \leq e^{B^{1/s}C}
\]
where $C$ is the constant in the definition of H\"{o}lder 
continuity.
\end{example}

\begin{lemma} \label{lemma:scicomp}
  Suppose that $\Psi$ is continuous and
  increasing and that its domain 
includes $[0,\infty)$. 
  If the modular family $\family{\Phi}$ is \sci, then so is 
$\family{\Phi} \circ \Psi$.
\end{lemma}
By $\family{\Phi} \circ \Psi$ we mean the family with function
$\Phi_x \circ \Psi$ at $x$. 
\begin{proof}
Let $B > 0$ and let $I$ be an interval and 
  suppose that $w \leq \Psi^{-1} \circ \Phi_{I,-}^{-1} (B/m(I))$.
  Let $z = \Psi(w)$.  Then $z \leq \Phi_{I,-}^{-1}(B/m(I))$. 
  Thus $\Phi_{I,+}(z) \leq C \Phi_{I,-}(z)$ so
  $\Phi_{I,+} \circ \Psi (w) \leq C \Phi_{I,-} \circ \Psi(w)$. 
\end{proof}

\begin{lemma}\label{lemma:powersci}
  If the modular family $\family{\Phi}$ is \sci\ so is the family
  $\Phi^p$ for any $p > 1$.
\end{lemma}
\begin{proof}
  Note that $m(I) \leq 2 \pi$ so
  $m(I) /(2 \pi) \leq 1$ and so
  $m(I)^{1/p} / (2 \pi)^{1/p} \geq m(I)/(2 \pi)$, and so
  $m(I)^{1/p} \leq m(I)/ (2\pi)^{1-1/p}$.
  Therefore
  $1/m(I)^{1/p} \leq (2\pi)^{1-1/p} / m(I)$.  
  If $\Phi(w)^p \leq B/m(I)$ then
\[
  \Phi(w) \leq B^{1/p}/m(I)^{1/p} \leq
  B^{1/p}(2\pi)^{1-1/p} / m(I).
\]
  Also
  \[
    \frac{\Phi_{+,I}(\lambda)^p}{\Phi_{-,I}(\lambda)^p} =
    \left(\frac{\Phi_{+,I}(\lambda)}{\Phi_{-,I}(\lambda)}\right)^p
    \leq C^p
  \]
  where the $C$ in question is the same as the constant in the
  definition of \sci\ for $\Phi$ with 
  $B^{1/p}(2\pi)^{1-1/p}$ in place of $B$.  
\end{proof}

\begin{example}
Suppose that $\Phi_x(t) = t^p w(x)$.  This corresponds to a weighted 
$L^p$ space.  Then the family $\Phi$ is slowly changing on intervals 
if and only if $w(x)$ is bounded above and below away from $0$, i.e.\ 
there are constants $c$ and $C$ such that $0 < c < w(x) < C < \infty$ for all 
$x$.  
\end{example}

As shown by the example above, 
the condition of slowly changing on intervals is sometimes 
too strict to be of interest.
We will thus introduce other conditions later.

We now come to definitions that will be important for stating
more general conditions than \sci. 
\begin{definition}            
Recall that an open set on the unit circle can be decomposed uniquely
into a countable disjoint union of open intervals (i.e. arcs).
We call each of these
intervals basic intervals for the open set.  
\end{definition}

\begin{definition}
For any $\lambda > 0$ and family $\family{\Phi}$, define the measure 
$\mu_{\Phi, \lambda} = \Phi_x(\lambda) \, dx$.  We sometimes use the
notation $\mu_{\lambda}$ if the family $\family{\Phi}$ is clear. 
\end{definition}

Recall that $\fint_I f \, dx$ is the average of the function
$f$ on the interval $I$. 
\begin{definition}
For any interval I, define 
\[
\phi_I (t) = \fint_I \phi_x(t) \, dt.
\]
\end{definition}

We now come to a technical lemma we will use in the future.
\begin{lemma}\label{lemma:geqtoleq}
  Let $\family{\Phi}$ be a modular family.  
Suppose that $\mathcal{F}$ is a family of pairs of functions 
such that if $(f,g) \in \mathcal{F}$ then 
$(\alpha f, \alpha g) \in \mathcal{F}$ for any $\alpha > 0$.  
Also suppose that for some 
$A>0$ there is a $B>0$ depending $A$ such that 
if $\rho(f) = A$ then $\rho(g) \geq B$.  
It follows that if $\rho(g) \leq B$ then either
$\rho(f) \leq A$ or $\rho(f) = \infty$.  It also follows that 
if $\rho(f)$ is finite, then $\|f\|_{(A)} \leq \|g\|_{(B)}$.  
\end{lemma}

\begin{proof}
Suppose that $A \leq \rho(f) < \infty$.
As we said above, there is a 
$c \geq 1$ such that $\rho(f/c) = A$.  Then 
$\rho(g/c) \geq B$, and so $\rho(g) \geq B$ since 
$c \geq 1$.  It follows that if $\rho(g) < B$ then 
$\rho(f) < A$ or $\rho(f) = \infty$.  Now, if 
$\rho(g) = B$ and $\rho(f)$ is finite, 
then $\rho(g/c) < B$ for all $c > 1$, 
and so $\rho(f/c) < A$ since $\rho(f/c) < \infty$.  
Letting $c \rightarrow 1$ 
and applying the dominated convergence theorem 
shows that $\rho(f) \leq A$.   

For the last part, suppose that $\|g\|_{(B)} = \lambda$.  
Then $\rho(g/\lambda) = B$ and so 
$\rho(f/\lambda) \leq A$, and so $\|f\|_{(A)} \leq \lambda$. 
\end{proof}


\begin{definition} We say that a modular family $\family{\Psi}$ is
  $L^1$ dominating if for every $A > 0$, there is an $A' > 0$
  such that if $\rho_{\Psi}(f) \leq A$ then $\|f\|_{L^1} \leq A'$.
\end{definition}

For example, a family is $L^1$ dominating if
there is a constant $K>0$ such that 
$\Psi_x(t) \geq Kt^b$ for some $b \geq 1$ and all $t$.  It is also
$L^1$ dominating if there are constants $K > 0$ and $b, c \geq 1$ such that
$\Psi_x(t) \geq Kt^b$ for $ 0 \leq t \leq 1$ and
$\Psi_x(t) \geq Kt^c$ for $t \geq 1$.

\section{Hardy Spaces}
Suppose that $f$ is analytic in $\mathbb{D}$, the unit disc.
Let $\family{\Phi}$ be a modular family. 
Let
\[
  \rho_{\Phi}(r,f) = \rho_{\Phi}(f(r e^{i\cdot})).
\]
Also define 
\[
  \rho_{H_{\Phi}}(f) = \sup_{0<r<1} \rho_{\Phi}(r,f).
\]
We may also define
\[
  M_{\Phi,(a)}(r,f) = \|f(r e^{i\cdot})\|_{\Phi, (a)}
\]
and similarly we may define
\[
  \|f\|_{H_{\Phi}, (a)} = \sup_{0<r<1} M_{\Phi,(a)}(r,f).
\]
We must be very careful, since a bound on $\rho(g)$ for a
family of functions $g$ does not imply a bound on
$\|g\|_{(C)}$.  Thus, it seems possible that
$\|f\|_{H_{\Phi}, (a)} = \infty$ for some $f$ such that
$\rho_{H_{\Phi}}(f) < \infty$.  (It is an open question, 
as far as we know,
to find an example of such a function, or determine
that such a function cannot exist).  However, we can say that
if $\rho_{H_{\Phi}}(f) = C$ then $\|f\|_{H_{\Phi},(C)} = 1$. 

Let $H_{\Phi}$ consist of all analytic $f$ defined in 
$\mathbb{D}$ such that $\rho_{H_\Phi}(f) < \infty$. 
We may similarly define $H_{\Phi,(a)}$ to consist of all 
analytic $f$ defined in $\mathbb{D}$ such that 
$\|f\|_{H_{\Phi},(a)} < \infty$.  As we have said above, it is 
an open question as far as we know whether these spaces coincide. 
In this paper we deal with the spaces $H_{\Phi}$. 
We may define similar $h_{\Phi}$ spaces for harmonic functions.

\section{Uneven Layer Cake Decomposition}

We first introduce a condition on modular families that we will need 
below.  
\begin{definition}
Suppose that for each $B > 0$, we can find constants 
$1 < C_1 < C_2 < \infty$ 
with the following property:
for any open interval $I$ and any $\lambda > 0$ such that 
$\mu_{\lambda}(I) \leq B$, 
we can find a 
$\lambda' > \lambda$ such that, for each $x \in I$, we have that 
\begin{equation}\label{eq:layerlambdaone}
C_1 \Phi_x (\lambda) \leq \Phi_x(\lambda') \leq C_2 \Phi_x(\lambda).
\end{equation}
Then we say $\Phi$ has the good ratio property. 
\end{definition}
Now suppose that
$\Phi$ has the good ratio property and that 
$\Psi$ is increasing and continuous and consider the family 
$\Phi \circ \Psi$.  Then 
$\mu_{\Phi, \Psi(\lambda)} = \mu_{\Phi \circ \Psi, \lambda}$. 
So if $\mu_{\Phi \circ \Psi, \lambda}(I) \leq B$ then 
$\mu_{\Phi, \Psi( \lambda)}(I) \leq B$.  Thus we can find a 
$\nu > \Psi(\lambda)$ such that 
\[
C_1 \Phi_x (\Psi(\lambda)) \leq \Phi_x(\nu) \leq C_2 
\Phi_x(\Psi(\lambda)).
\]
Then if we let $\lambda' = \Psi^{-1}(\nu)$, we have 
\[
C_1 \Phi_x (\Psi(\lambda)) \leq \Phi_x(\Psi(\lambda')) \leq C_2 
\Phi_x(\Psi(\lambda)).
\]
Thus $\Phi \circ \Psi$ has the good ratio property. 

Note that any basic interval of a level set $\{x: f(x) > \lambda\}$ of 
a nonnegative lower semicontinuous function $f$ with 
$\rho_{\Phi}(f) \leq B$ will satisfy 
the condition $\mu_{\lambda}(I) \leq B$.  


We now discuss the uneven layer cake decomposition. 
Let $f$ be a nonnegative function that is lower semicontinuous and 
on the unit circle, and suppose 
$\rho(f) \leq B$.
Recall that, for any 
$\lambda$, the set $\{x : f(x) > \lambda\}$ is open, and thus is the disjoint 
union of open intervals, which we will call basic intervals for the set.  
Now suppose $\Phi$ has the good ratio property.
Let $I$, $\lambda$, and 
$\lambda'$ be as in the definition of the good ratio property.
Let $I'$ be a basic interval in 
$\{x : f(x) > \lambda'\} \cap I$.  Note that $I'$ is also a basic interval in 
$\{x : f(x) > \lambda'\}$.  We will say that the interval $I'$ with height $\lambda'$ 
is a child of the interval $I$ with height $\lambda$.  It is possible that 
$I' = I$, so that an interval can be its own child, but with different heights.  

Now, consider the set $\{x: f(x) > 1\}$.  Find the children of each basic interval
$I$ in this set by choosing some $\lambda'$ for each $I$ so that equation 
\eqref{eq:layerlambdaone} holds with $\lambda = 1$.
Find children of each of these 
children in a similar manner.  Continue by induction.  Then we have a tree like 
structure.  Call it $\mathcal{T}$.
We will consider members of $\mathcal{T}$ to 
be intervals $I$ with associated heights $\lambda_I$.
Note that one interval may 
be in $\mathcal{T}$ with multiple heights.  

For each $x$, it will belong to intervals
$I_1(x)$, $I_2(x)$, $I_3(x)$, $\ldots$
where each $I_{j+1}(x)$ is the child of $I_j(x)$.  
Let $\lambda_j(x)$ be the height associated with the interval $I_j(x)$.  
Notice that
for almost every $x$, 
there is some $n$ so that 
$x$ is not in any child of $I_n(x)$. For suppose this is not the
case for some particular $x$.  Then 
$\Phi_x(\lambda_j(x)) \rightarrow \infty$ as
$j \rightarrow \infty$, which 
means that $\lambda_j(x) \rightarrow \infty$.
But if $x$ were in all the $I_j$, then
$f(x) > \lambda_j(x)$ for all $j$, and so $f(x)$ would be infinite.  
Define $\lv(x) = n$.  Notice that
\[
  \Phi_x(\lambda_{\lv x}) \leq \Phi_x(f(x)) \leq
  \Phi_{x}(\lambda_{\lv(x) + 1}) \leq C_2 \Phi_x(\lambda_{\lv x}).
\]
Let $\rc(I)$ denote $I$ with all of its children removed. 

Now notice that $\Phi_x(\lambda_j) \leq C_1^{-(n-j)} \Phi_x(\lambda_n)$, so that 
\[
\Phi_x(\lambda_n) \leq \sum_{j=1}^n \Phi_x(\lambda_j) \leq 
\frac{C_1}{C_1-1} \Phi_x(\lambda_n)
\]
by the formula for the sum of an infinite
geometric series with ratio $1/C_1$.

Integrating the right hand side of the above inequality gives 
\[
\begin{split}
\sum_{I \in \mathcal{T}} \mu_{\lambda_I}(I) &\leq
\frac{C_1}{C_1-1} \sum_{I \in \mathcal{T}} \mu_{\lambda_I}(\rc (I)) \\
&\leq
\frac{C_1}{C_1-1} \int_0^{2\pi} \Phi_x(\lambda_{\lv(x)}(x)) \, dx \\
&\leq 
\frac{C_1}{C_1-1} \int_0^{2\pi} \Phi_x(f(x)) \, dx.
\end{split}
\]
Note that $\lambda_{\lv(x)}$ is defined for almost every $x$, so the
inequality above makes sense. 

Also 
\[
\begin{split}
 \int_0^{2\pi} \Phi_x(f(x)) \, dx &\leq
\int_0^{2\pi} \Phi_x(\lambda_{\lv(x)+1}(x)) \, dx 
\leq
C_2 \int_0^{2\pi} \Phi_x(\lambda_{\lv(x)}(x)) \, dx 
\\
&=
C_2 \sum_{I \in \mathcal{T}} \mu_{\lambda_I}(\rc(I)) 
\leq
C_2 \sum_{I \in \mathcal{T}} \mu_{\lambda_I}(I).
\end{split}
\]

And thus we have the following theorem, which we call the modular inequality 
for the uneven layer cake decomposition:
\begin{theorem}
If the family $\Phi$ has the good ratio property, then 
for every function $f$ with $\rho(f) \leq B$ we have that 
\[
\frac{C_1-1}{C_1} \sum_{I \in \mathcal{T}} \mu_{\lambda_I}(I) \leq 
\int_0^{2\pi} \Phi_x(f(x)) \, dx 
\leq 
C_2 \sum_{I \in \mathcal{T}} \mu_{\lambda_I}(I),
\]
where the constants $C_1$ and $C_2$ are as in the definition of 
the good ratio property.
\end{theorem}

Notice that to use the uneven layer cake decomposition, we 
first assume that $\rho(f) \leq B$ for some $B$, and that the 
constants $C_1$ and $C_2$ depend on $B$.  Thus, it may seem 
that this decomposition is not useful for bounding 
modulars of functions, since we have to assume the modular 
is bounded in order to use it!

There are two ways around this.   
First of all, if the constants $C_1$ and $C_2$ are independent of 
$B$, then the 
modular inequality for the uneven layer cake decomposition 
holds for any $f$ with finite modular, and we may
use the same numbers $\lambda_1, \lambda_2, \ldots$ for 
each point $x$.  Thus, the decomposition can be made in to an 
``even'' layer cake decomposition, where the numbers 
$\lambda_j$ are the same no matter what the function
$f$ is.  The finiteness assumption on 
the modular of $f$ may be removed by the monotone 
convergence theorem.
These above remarks apply, for example, 
if the space in question is a 
weighted variable $L^{p(\cdot)}$ space where 
$0 < p_{-} < p_{+} < \infty$.

The second way around the problem with the 
uneven layer cake decomposition is more complicated, 
but can be used if 
the numbers $\lambda_1(x), \lambda_2(x), \ldots$ cannot 
be chosen independently of $x$.  First, as we observe in 
equation \eqref{eq:layercakeotherfunction}, if we have 
another function $g$, then a modified form of the left hand 
side of the modular inequality holds for $g$, but still 
using the uneven layer cake decomposition of $f$.  
See below for details.  If we have two functions $f$ and 
$g$ related in certain ways (for example, by a 
good-$\lambda$ inequality), we can sometimes use this observation 
to prove that for every $A >0$, there is a $B > 0$ such that 
if $\rho(f) = A$ then 
$\rho(g) \geq B$.  By Lemma \ref{lemma:geqtoleq}, this often 
implies that if $\rho(g) \leq B$ then $\rho(f) \leq A$, if 
$\rho(f)$ is finite.  This finiteness assumption can often 
be removed by an approximation argument. 

\begin{example}
Suppose that $\Phi_x(t) = t^{p(x)} w(x)$, so that the space is a weighted 
variable Lebesgue space.  If $p_{-} = \essinf p(x)$ and $p_{+} = \esssup p(x)$ 
have the property that 
\[
0 < p_{-} < p_{+} < \infty\]
 then we may take 
$C_1 = 2^{p_{-}}$ and $C_2 = 2^{p_{+}}$ and we can choose 
$\lambda_n = 2^{n-1}$ (irrespective of $x$).
\end{example}

As we see in the above example, in many important cases
we can choose the numbers $\lambda_j$ independently of $x$.
This case is much simpler and makes the decomposition below
simpler to apply.  However, this does not work in the case of
variable exponent spaces where $p_{-} = 0$, or with spaces
as in Example \ref{ex:logsci}. 

\begin{example}
If the family $\Phi$ has the good ratio property, then so 
does the family $\Psi$ given by $\Psi_x(t) = \Phi_x(t) w(x)$ for any function 
$w(x) > 0$.  For the family $\Psi$, 
we can choose the same constants $C_1$ and $C_2$ and the same 
$\lambda'$ for each interval $I$ with height $\lambda$ as we did for 
$\Phi$.  
\end{example}

\begin{example}
If a family $\Phi$ is slowly changing on intervals, then it has 
the good ratio property.
Given a constant $B > 0$, let $I$ be such that 
$\mu_\lambda(I) \leq B$.  
This means that $\Phi_{I,-}(\lambda) \leq B/m(I)$, and 
so $\lambda \leq \Phi_{I,-}^{-1}(B/m(I))$.  
Thus $\Phi_{I,+}(\lambda) \leq C \Phi_{I,-}(\lambda)$ for 
the constant $C$ in the definition of slowly changing on 
intervals.  Now choose $\lambda'$ so that 
$\Phi_{I,-}(\lambda') = A \Phi_{I,-}(\lambda)$ for some 
$A > C$.  It follows 
for each $x \in I$ that 
\[
\Phi_x(\lambda') \geq \Phi_{I,-}(\lambda') = A \Phi_{I,-}(\lambda) 
\geq (A/C) \Phi_{I,+}(\lambda) \geq (A/C) \Phi_x(\lambda)
\]
and 
\[
\Phi_x(\lambda') \leq \Phi_{I,+}(\lambda') \leq 
C \Phi_{I,-}(\lambda') \leq A C \Phi_{I,-}(\lambda) 
\leq A C \Phi_x(\lambda).
\]
\end{example}

We now make two comments.  First of all, suppose that $h$ is another 
lower semicontinuous 
nonnegative function, 
and that the inequality $h(x) \leq f(x)$ holds pointwise.  
Note that $\{x: h(x) > \lambda\} \subset \{x: f(x) > \lambda\}$.  
Then we may 
form an uneven layer cake decomposition of $h$ using the same numbers 
$\lambda_n(x)$ for each point $x$, where each interval 
in the decomposition for $h$ is a subset of an interval in the 
decomposition for $f$ with the same associated height $\lambda$.

Secondly, suppose that $g$ is another function.  
For each x, define $\lv_g(x)$ to be the largest $j$ such that 
$0 \leq j \leq \lv(x)$ and 
$\lambda_j(x) \leq g(x)$.  
Then the sum 
\begin{equation}\label{eq:layercakeotherfunction}
\begin{split}
\sum_{I \in \mathcal{T}} \mu_{\lambda_I}(I \cap \{x: g(x) > \lambda_I\})
&=
\int_0^{2\pi} \sum_{n=1}^{\lv_g(x)} \Phi_x(\lambda_n) \, dx \\
&\leq 
\frac{C_1}{C_1-1} \int_0^{2\pi} \Phi_x(\lambda_{\lv_{g}(x)}) \, dx \\
&\leq 
\frac{C_1}{C_1-1} \int_0^{2\pi} \Phi_x(g(x)) \, dx.
\end{split}
\end{equation}

\begin{definition}
We will say that the family $\Phi$ has the $A_\infty$ property if 
for each $B > 0$
there are constants $\epsilon$ and $\delta$ such that 
$0 < \epsilon, \delta < 1$ and for which the following holds: 
Let $I$ be any interval such that 
$\mu_{\lambda}(I) \leq B$ and 
let $E \subset I$ such that $m(E) < \delta m(I)$.  Then 
$\mu_\lambda(E) < \epsilon \mu_\lambda(I)$ for each 
$\lambda$. 
\end{definition}
In other words, this definition says that, when restricted to 
small enough intervals (depending on $\lambda$), each 
measure $\mu_{\lambda}$ is an $A_\infty$ measure, with 
uniform constants.

It is important here that we restrict to small intervals
when considering large $\lambda$.  For example, in a variable
exponent Lebesgue space, the measures $\mu_\lambda$ cannot all
be $A_\infty$ on the whole circle unless the exponent is constant
almost everywhere. 

Suppose that $\Psi$ is an increasing function and 
$\Phi$ is $A_\infty$.  Suppose that 
$\mu_{\Phi \circ \Psi, \lambda}(I) \leq B$.  This means that 
$\mu_{\Phi, \Psi(\lambda)} \leq B$.  
Thus there are constants $\epsilon$ and $\delta$ such that 
$0 < \epsilon, \delta < 1$ and if 
$E \subset I$ and $m(E) < \delta m(I)$ then 
$\mu_{\Phi, \Psi(\lambda)}(E) < \epsilon \mu_{\Phi, \Psi(\lambda)}(I)$.
But this implies that 
$\mu_{\Phi \circ \Psi, \lambda}(E) 
< \epsilon \mu_{\Phi \circ \Psi, \lambda}(I)$.
Thus, $\Phi \circ \Psi$ also has the $A_\infty$ property.

By basic facts about $A_\infty$ measures, for every $\epsilon$ we 
can find a $\delta$ that makes the $A_\infty$ property hold for intervals 
$I$ such that $\mu_\lambda(I) \leq B$.  
Also, there is some $p$ such that, when restricted to 
such intervals, each measure $\mu_\lambda$ is 
in $A^p$ for some $p$, and the $A^p$ norms are uniformly bounded.
Also, applying the $A_\infty$ condition to
$E^c$ shows that
if $E \subset I$ and $m(E) \geq (1 - \delta) m(I)$, then
$\mu_\lambda(I) \geq (1 - \epsilon) m(I)$. 

\begin{example}
Suppose that $\Phi$ is slowly changing on intervals.  
Suppose that $\mu_\lambda(I) \leq B$ for some 
$B > 0$.  Then note that 
\[
\Phi_{I,-}(\lambda) \leq \Phi_x(\lambda) \leq \Phi_{I,+}(\lambda) 
\leq C \Phi_{I,-}(\lambda),
\]
where $C$ is as in the definition of slowly changing on intervals.
This shows that 
\[
\Phi_{I,-}(\lambda) m(A) \leq \mu_\lambda(A) 
\leq C \Phi_{I,-}(\lambda) m(A)
\]
for any $A \subset I$, where $m(A)$ is Lebesgue measure.  
But this means that if $m(A) \leq \delta m(I)$ then 
\[
\mu_{\lambda}(A) \leq \frac{C \delta}{1 + (C-1)\delta}.
\]
Thus $\Phi$ has the $A_\infty$ property. 
\end{example}

We now prove the following theorem:
\begin{theorem} \label{thm:twofunctionmeasureinequality} 
Let $f$ and $g$ be nonnegative functions.  
Suppose that there is a constant $0 < \alpha < 1$ such that 
for every $\lambda > 0$ we have that 
\[
m( \{x: g(x) > \lambda\} \cap I) > \alpha m( I)
\]
for each interval $I$ that is a fundamental interval in 
$\{x: f(x) > \lambda\}$.  Assume
the family $\Phi$ has goood ratio property and the 
$A_\infty$ property.  
Then for each $A$, there is a $B$ such that if 
$\rho(f) = A$ then $\rho(g) \geq B$. 
\end{theorem}

\begin{proof}
 Let $\rho(f) = A$.
 Let $\delta = 1 - \alpha$ and let $\epsilon$ be such that if
 $m(E) < \delta m(I)$ then $\mu_{\lambda}(E) < \epsilon \mu_\lambda(I)$
 for any $E \subset I$, where $I$ is any interval with
 $\mu_{\lambda}(I) < A$.  
 Form the uneven layer cake decomposition of $f$. 
 By the modular inequality for the uneven layer cake decomposition, 
 we have that 
\[
 \int_0^{2\pi} \Phi_x(f(x)) \, dx 
\leq 
C_2 \sum_{I \in \mathcal{T}} \mu_{\lambda_I}(I) 
\leq
C_2 (1 - \epsilon)
  \sum_{I \in \mathcal{T}} \mu_{\lambda_I}(I \cap \{x:g(x) > \lambda\}). 
\]
By inequality \eqref{eq:layercakeotherfunction}, we have 
\[
 \int_0^{2\pi} \Phi_x(f(x)) \, dx 
 \leq
C_2 (1 - \epsilon) \frac{C_1}{C_1-1} \int_0^{2\pi} \Phi_x(g(x)) \, dx.
\]
 \end{proof}

Now suppose that we have three functions, $f$, $g$, and $h$, and that $f$ 
and $h$ are lower semicontinous and $h(x) \leq f(x)$ pointwise.
Also assume that 
$\rho(f) = A > 0$.  Let 
$\mathcal{T}$ be the collection of intervals (with associated levels) in the 
uneven layer cake decomposition of $f$.  Suppose that there is a fixed 
$\delta < 1$ such that, for each 
$I \in \mathcal{T}$, we have that 
\[
m(\{ x \in I : h(x) > \lambda \text{ and } g(x) < \lambda \}) < \delta m(I).
\]
This implies, by the $A_\infty$ condition, that for each $I \in \mathcal{T}$ 
we have that 
\[
\mu_{\lambda_I} (\{ x \in I : h(x) > \lambda \text{ and } g(x) < \lambda \}) < 
       \epsilon \mu_{\lambda_I}(I).
\]
Note that $\epsilon$ depends on $A$.  
Now notice that 
\[
\mu(\{ x \in I : h(x) > \lambda\}) \leq 
\mu(\{ x \in I : h(x) > \lambda \text{ and } g(x) < \lambda \}) + 
\mu(\{x \in I : g(x) > \lambda\}).
\]
for any measure $\mu$. 

Thus 
\[
\begin{split}
\int_0^{2\pi} \Phi_x(h(x)) \, dx &\leq 
C_2 \sum_{I \in \mathcal{T}} \mu_{\lambda_I}(I \cap \{x: h(x) > \lambda\}) \\
&\leq 
C_2 \sum_{I \in \mathcal{T}}  \big[
\mu_{\lambda_I}(\{ x \in I : h(x) > \lambda \text{ and } g(x) < \lambda \}) + \\
& \qquad \mu_{\lambda_I}(\{x \in I : g(x) > \lambda\}) \big]
\\
&\leq C_2 \sum_{I \in \mathcal{T}} \big[ \epsilon \mu_{\lambda_I}(I) +
  \mu_{\lambda_I}(\{x : g(x) > \lambda\} \cap I) \big] 
\\
&\leq
\epsilon C_2   \frac{C_1}{C_1-1} \int_0^{2\pi} \Phi_x(f(x)) \, dx + 
C_2 \frac{C_1}{C_1 - 1} \int_0^{2\pi} \Phi_x(g(x)) \, dx.
\end{split}
\]
Remember here that $C_1$ and $C_2$ depend on $A$.  

We now consider the case when the inequality above is
\[
m(\{ x \in I : f(x) > \beta \lambda \text{ and } g(x) < \gamma \lambda \}) < \delta m(I),
\]
for each basic interval $I$, where $\beta > 1$.  In other words 
\[
m(\{ x \in I : f(x)/\beta  > \ \lambda \text{ and } g(x)/\gamma < \lambda \}) < \delta m(I).
\]
We make the assumption that 
there is a family $\mathcal{F}$ of pairs of functions 
such that if $(f,g) \in \mathcal{F}$ then 
$(\alpha f, \alpha g) \in \mathcal{F}$ for any $\alpha > 0$.  

This then implies that 
\[
\int_0^{2\pi} \Phi_x(f(x)/\beta) \, dx \leq
\epsilon C_2   \frac{C_1}{C_1-1} \int_0^{2\pi} \Phi_x(f(x)) \, dx + 
C_2 \frac{C_1}{C_1 - 1} \int_0^{2\pi} \Phi_x(g(x)/\gamma) \, dx.
\]

Now, if $\Phi$ is uniformly doubling, then there exists a constant 
$D_\beta$ depending on $\beta$ such that
$\Phi_x(a) < D_\beta \Phi_x(a / \beta)$ for any $a$.
Define $D_{1/\gamma}$ similarly.  Then we have that 
\[
\int_0^{2\pi} \Phi_x(f(x)) \, dx \leq 
\frac{\epsilon D_\beta C_1 C_2}{C_1 - 1} \int_0^{2\pi} \Phi_x(f(x)) \, dx + 
\int_0^{2\pi}
\frac{ D_{1/\gamma} D_\beta C_1 C_2}{C_1 - 1} \Phi_x(g(x)) \, dx.
\]

Now, suppose that for some fixed $\beta$, 
we can make $\delta$ as small as we wish by choosing $\gamma$ very small.  
Also recall that we can make $\epsilon$ as small as we wish by making 
$\delta$ very small. Then we have that 
\[
\frac{C_1-1}{ D_{1/\gamma} D_\beta C_1 C_2}
 \left[ 1 - \frac{\epsilon D_\beta C_1 C_2}{C_1 - 1}\right] 
 \rho(f) \leq \rho(g).
 \]
 Call the left side of the above inequality $B$.  
The constant
\[
\frac{C_1-1}{ D_{1/\gamma} D_\beta C_1 C_2}
 \left[ 1 - \frac{\epsilon D_\beta C_1 C_2}{C_1 - 1}\right] 
 \]
does depend on $A$, and we can make it positive by making 
$\epsilon$ as small as we wish.  
 
 This means that if $\rho(f) = A$ then $\rho(g) \geq B$.  
By Lemma \ref{lemma:geqtoleq}, we have the following theorem.

\begin{theorem}\label{thm:goodlambda}
Suppose that the family $\Phi$ is uniformly doubling and has the 
$A_\infty$ property. 
Suppose that $\mathcal{F}$ is a family of pairs of functions 
such that if $(f,g) \in \mathcal{F}$ then 
$(\alpha f, \alpha g) \in \mathcal{F}$ for any $\alpha > 0$, and that 
the following good lambda inequality holds between any 
$f$ and $g$ such that $(f,g) \in \mathcal{F}$:
There is a $\beta > 0$ such that 
for each sufficiently small $\delta > 0$ there exists a $\gamma > 0$
such that 
for each $\lambda > 0$ and  basic interval $I$ in the level set of 
$\{x: f(x) > \lambda\}$ one has
\[
  m(\{ x \in I : f(x) > \beta \lambda \text{ and }
    g(x) < \gamma \lambda \}) < \delta m(I).
\]
Then for each $A > 0$, there is a $B > 0$ such that if
$\rho_\Phi(f)$ is finite then  
$\|f\|_{(A)}\leq \|g\|_{(B)}$.  Also if 
$\rho(g) \leq B$ then $\rho(f) \leq A$ or $\rho(f) = \infty$.  
\end{theorem}

We briefly mention how to do an integral type of layer cake 
decomposition, and how to use it to prove a theorem similar to the above.  
Suppose for simplicity that each $\Phi_x'$ is $C^1$ for each $x$.  

By Fubini's theorem, we have that 
\[
\int_0^{2\pi} \Phi_x(f(x)) \, dx = 
\int_0^\infty \mu'_\lambda (\left\{x: f(x) > \lambda\right\}) \, d \lambda
\]
where $\mu'_\lambda = \Phi'_x(\lambda) \, dx$. 

The advantage of this layer cake decomposition over the other is that 
it is an equality, and that there are no constants $C_1$ and $C_2$ 
that we must control, depending on 
the modular of $f$.  
The disadvantage is that it 
involves derivatives of $\Phi_x$ instead of the functions $\Phi_x$ themselves.  

We could  say that the family $\Phi$ has the 
$A'_\infty$ property if for 
each $B> 0$, the measures $\mu'_{\lambda}$ are uniformly 
$A_\infty$ when restrited to intervals
$I$ such that $\mu_{\lambda}(I) < B$.  
If $\Phi$ has this property, we can prove a theorem 
about good-$\lambda$ inequalities similar to Theorem \ref{thm:goodlambda}, 
but using the integral layer cake decomposition instead.  Note however that 
if $\Phi$ has the $A'_\infty$ property than 
Fubini's theorem shows that it has 
the $A_\infty$ property.

\section{An Application to Harmonic Functions}
It is known that  
the assumptions of Theorem \ref{thm:goodlambda} are 
satisfied if we take $f$ equal to the Lusin area integral of a harmonic 
function and $g$ equal to the nontangential maximal function of a harmonic 
function.  It also holds with the places of the area
integral and the nontangential 
maximal function reversed
\cite{Fefferman-Gundy-Silverstein-Stein_good-lambda}.

A similar argument to the proof of Theorem 3 in
\cite{Fefferman-Gundy-Silverstein-Stein_good-lambda}
(where instead of Green's theorem we use
Cauchy's integral theorem from complex analysis) shows that
a good-$\lambda$ inequality of the type required holds between the
nontangential maximal function of a harmonic function and
the nontangential maximal function of its harmonic conjugate. 
Thus, Theorem \ref{thm:goodlambda} applies in all of these
cases.  

In this context, we do not need the assumption of finite
norm in Theorem \ref{thm:goodlambda}.  To see this, let 
$u$ be harmonic in the unit disc and 
let $\alpha(z) = u(z) + i v(z)$ be its analytic completion.
Let $\alpha_r$ be the dilation of
$\alpha$ for $0 < r < 1$.  Then $u_r^* \leq C u^*$ and
certainly $v_r^*$ is in $L_\Phi$ since $v_r^*$ 
is bounded.  Also, the Lusin area 
integral $A u_r$ is uniformly continuous and
$A u_r \leq C A u$ for $r$ close to $1$, where $C$ is 
a constant. 
Thus for each $C>0$, there is a $B > 0$ such that 
$\|A u_r\|_{(C)} \leq \|u_r^*\|_{(B)}$.
Letting $r \rightarrow 1$ gives a result for 
$A u$ and $u^*$.  A similar result holds for 
$u^*$ and $v^*$.  

A technicality we have not mentioned is that 
we may have to choose different apertures of cones for these
good-$\lambda$ inequalities to hold. However, this does not matter.
To see this, let $u^*_a$ and $u^*_b$ be maximal functions of $u$
using cones of apertures $a$ and $b$ respectively, where $b > a$.
Then every basic interval $I$ in a level set of $u^*_b$ contains basic
intervals $J_k$ in a corresponding level set of $u^*_a$, and the total
measure of all the $J_k$ is greater than some (possibly small) constant
times the measure of $I$.  This can be proven 
using
an argument similar to that in 
Section II.2.5 in \cite{Stein_harmonicanalysis}, or using the 
Vitali covering lemma.  In the latter argument, one applies the 
Vitali covering 
lemma to the collection of intervals
$\{T_b(z) \cap \mathbb{D} : u(z) > \lambda\}$, where 
$T_b(z)$ is the tent with peak at $z$.  Note that 
the length of the base of $T_a(z)$ is greater than some constant 
times the length of the base of $T_b(z)$. 

Because 
every basic interval $I$ in a level set of $u^*_b$ contains basic
intervals $J_k$ in a corresponding level set of $u^*_a$, and the total
measure of all the $J_k$ is greater than some (possibly small) constant
times the measure of $I$, 
we have by Theorem \ref{thm:twofunctionmeasureinequality} 
that for each 
$A > 0$ such that $\rho(u^*_b) = A$, there is a $B > 0$
such that $\rho(u^*_a) \geq B$.  
But this implies, by Lemma \ref{lemma:geqtoleq}, that 
for every $A > 0$, there is a $B > 0$ such that 
$\|u^*_b\|_{(A)} \leq \|u^*_a\|_{(B)}$.
The requirement from the lemma that the left 
side of the last inequality be finite can be removed 
by a dilation argument as above.  

%


We may thus state the following theorem.  
\begin{theorem}
Suppose that $f = u + iv$ is analytic in the unit disc, and let the 
modular family $\family{\Phi}$ be \udoub\ and
have the $A_\infty$ property and the good ratio property.  Then the
following are equivalent:
\begin{enumerate}
\item $u^* \in L_{\phi}$
\item $v^* \in L_{\phi}$
\item $Au \in L_{\phi}$
\end{enumerate}
Furthermore, we have that for every $C > 0$, there is a 
$B > 0$ such that  
\[ 
\|f\|_{(C)} \leq \|g\|_{(B)}
\]
where $f$ and $g$ can be any of the functions
$u^*$, $v^*$, or $Au$.  
\end{theorem}
Notice that the inequality between $u^*$ and $v^*$ follows
from the inequality between $u^*$ and $Au$ and that between
$Au$ and $v^*$.

\section{The Jensen property}

\begin{definition}
We say that a family $\Phi$ 
has the Jensen property if for each $B > 0$, there 
is a constant $C$ such that for all intervals $I$ and all nonnegative 
$f$ with $\rho(f) \leq B$ we have
\[
\phi_I\left(\fint_I f(x) \, dx\right) \leq C \fint_I \phi_x(f(x)) \, dx.
\]
\end{definition}

If $\Phi_x(t) = \Phi(t)$ is a convex function, this holds by 
Jensen's inequality.  In fact, by the proof of Jensen's inequality,
if $\Phi(t)$ is a function such that there is some
$C > 0$ such that the graph of $C \Phi(t)$ lies above every tangent line
to the graph of $\Phi(t)$, then $\Phi(t)$ has the
Jensen property.  As a consequence, if there is some
$a > 0$ such that $\Phi(0) = 0$ and $\Phi(a) > 0$, and such that
$\Phi$ is linear between $t=0$ and $t=a$ and convex in the
region $t \geq a$, then the family $\Phi_x(t) = \Phi(t)$
has the Jensen property.

Observe that if the family $\Phi$ has the Jensen property and 
the function $\alpha$ is convex, then the family 
$\Phi \circ \alpha$ has the Jensen property as well.  To see this, 
observe that
\[
  (\Phi \circ \alpha)_I(\lambda) =
  \fint_I \Phi_x(\alpha(\lambda)) \, dx =
  \Phi_I(\alpha(\lambda))
\]
so $(\Phi \circ \alpha)_I = \Phi_I \circ \alpha$.  
Suppose that $\rho_{\Phi \circ \alpha}(f) \leq B$.  Then 
$\rho_{\Phi}(\alpha(f)) \leq B$ so 
\[
\Phi_I \circ \alpha \left(\fint_I f(x) \, dx\right) \leq 
\Phi_I \left(\fint_I \alpha(f(x)) \, dx\right) \leq 
C  \fint_I \Phi_x(\alpha(f(x))) \, dx.
\]

\begin{theorem} If the family $\Phi$ is slowly changing on intervals, and if 
$\Phi_{I,-}$ is convex for each interval $I$,  
then the family has the Jensen property.
\end{theorem}

\begin{proof}
We prove a stronger inequality.  Suppose that 
$\rho(f) \leq A$.  Now notice that 
the average $\fint_{I} \Phi_{I,-}(f(y)) \, dy$ is at most
$A/m(I)$.  

So by Jensen's inequality,
\[
  \Phi_{I,-} \left(\fint_{I} f(y) \, dy \right) \leq A/m(I).
\]
By the condition of slowly changing on intervals, we have 
\[
  \Phi_{I,+} \left(\fint_{I} f(y) \, dy \right) \leq C \Phi_{I,-} \left(\fint_{I} f(y) \, dy \right).
\]
Thus 
\[
\begin{split}
\Phi_{I} \left(\fint_{I} f(y) \, dy \right) &\leq 
\Phi_{I,+} \left(\fint_{I} f(y) \, dy \right) \\ &\leq
C \Phi_{I,-} \left(\fint_{I} f(y) \, dy \right) \\ &\leq
C \fint_{I} \Phi_{I,-}(f(x)) \, dx \\ &\leq
C \fint_{I} \Phi_{x}(f(x)).
\end{split}
\] 

\end{proof}

\begin{example}
  If $\Phi_x(t) = t w(x)$, where $w(x) \in L^1$,
  we have that the Jensen condition becomes
\[
\frac{w(I)}{m(I)} \fint_I f(x) \, dx \leq C \fint f(x) w(x) \, dx.
\]
Now the infimum of the right hand side, among all functions with a 
given average on $I$, is 
\[
\left( \fint_I f(x) \, dx \right) \essinf_{x \in I} w(x).
\]
So the Jensen property is equivalent to 
\[
\frac{w(I)}{m(I)} \leq C \essinf_{x \in I} w(x)
\]
which is the $A_1$ condition on weights.  
\end{example}

\begin{example}
If instead we have $\phi_x(t) = t^p w(x)$, the property becomes
\[
  \frac{w(I)}{m(I)} \left( \fint_I f(x) \, dx \right)^p
   \leq C \fint f(x)^p w(x) \, dx
\]
which becomes
\[
  \left(\fint_I f(x) \, dx \right)^p
  \leq C \frac{1}{w(I)} \int_I f(x)^p w(x) \, dx.
\]
This is equivalent to the $A_p$ condition.
\end{example}

We now show how to test the Jensen property in a large 
number of cases.  
Suppose that each $\Phi_x(t)$ is in 
$C^1$ with respect to $t$, and continuous in $x$, and that 
$\Phi_x'$ is continuous in $t$ and $x$ 
and strictly increasing for each $x$.  
Let $I$ be an interval in 
question.  Suppose that there is some $f$ with 
\[
\fint_I \Phi_x(f(x)) \, dx = A
\]
and 
\[
\Phi_I \left( \fint_I f(x) \, dx \right) > B.
\]
Then there is some $n$ and some function $f_1$ that is constant on dyadic 
subintervals of $I$ of size $2^{-n} | I |$ such that 
\[
\fint_I \Phi_x(f_1(x)) \, dx = A
\]
and 
\[
\Phi_I \left( \fint_I f_1(x) \, dx \right) > B.
\]
Let $y_j$ be the value of $f_1$ on the $j^{th}$ subinterval.  The condition
\[
\fint_I \Phi_x(f_1(x)) \, dx = A,
\]
or equivalently 
\begin{equation}\label{eq:jensentestnormconstraint}
 \frac{1}{2^n} \sum_{j=1}^{2^n} \Phi_{I_j}(y_j) = A
\end{equation}
means that the set of possible $y_j$ is compact, and thus 
\[
\Phi_I \left( \fint_I f_1(x) \, dx \right) = 
\Phi_I \left( \frac{1}{2^n}\sum_{j=1}^{2^n} y_j \right)
\]
attains a maximum.  
By the method of Langrange multipliers, the maximum must occur when 
\[
\Phi'_I \left(\frac{1}{2^n} \sum_{j=1}^{2^n} y_j \right) = 
\lambda \frac{1}{2^n} \Phi'_{I_j}(y_j)
\]
for each $j$, where $\lambda$ is some constant.
It is clear that $\lambda$ cannot be $0$.  
Now notice that the left hand side is constant 
in $j$, and so we have that 
\[
y_j = \Phi'^{-1}_{I_j}(c_n)
\]
for some constant $c_n$ independent of $j$, 
chosen so that the condition \eqref{eq:jensentestnormconstraint}
is satisfied.  Note that for each $c_n$, this condition 
gives a unique solution for the $y_j$, because 
$\Phi_{I_j}'$ is strictly increasing. 
Notice that the maximum value of the quantity we are maximizing 
cannot decrease with $n$, since every function that is constant on 
dyadic subintervals for a given $n$ is also constant on those intervals 
for larger $n$.  

Suppose for now that there are two continuous,
strictly increasing functions
$E(t)$ and 
$D(t)$ such that $E(t) \leq \Phi_x'(t) \leq D(t)$ for each $x$, and that 
$\Phi_{I,-}(t)$ is a strictly increasing function of $t$.
We can generally remove these assumptions later by 
approximation arguments. 
Observe that the left side of 
\eqref{eq:jensentestnormconstraint} is at least 
$\Phi_{I,-}(D^{-1}(c_n))$ and so 
$c_n < D(\Phi_{I,-}^{-1}(A))$.  
Also, each $y_j$ is at most 
$E^{-1}(D(\Phi_{I,-}^{-1}(A)))$. 

%

Since the $c_n$ are bounded, there is
 a sequence $n_j$ for which $c_{n_j}$ converges to some 
$c$.  Notice that, for a 
given $x$, if we let $I_{n}(x)$ be the dyadic subinterval of size 
$2^{-n} | I |$ to which $x$ belongs, that 
$\Phi'_{I_n(x)}(t)$ approaches $\Phi'_x(t)$ for fixed $t$.  
Then 
$\Phi'^{-1}_{I_n (x)}(t)$ approaches $\Phi'^{-1}_x(t)$ for
fixed $t$.  

Let $f_j$ be the function that equals  
$\Phi'^{-1}_{I_k}(c_{n_j})$ on the dyadic interval 
$I_k$ of size $2^{-n_j}$. 
Then as $j \rightarrow \infty$, the functions 
$f_j$ approach 
\[
f(x) = \Phi'^{-1}_x(c)
\]
pointwise.
Note that $f$ and all the $f_j$ are dominated pointwise by
$E^{-1}(D(\Phi^{-1}_{I,-}(A)))$.  
By the dominated convergence theorem, one has that
\[
\fint_I \Phi_x(f(x)) \, dx = A \quad \text{and} \quad
\Phi_I\left( \fint_I f(x) \, dx \right) > B.
\]
From this it follows that we need only check functions of the 
form
\[
f(x) = \Phi'^{-1}_x(c)
\]
in the Jensen property. 

\begin{example}
Let $\Phi_x(t) = w(x) t^p$.  We will for now assume that 
$w$ is continuous and bounded above and below, although these 
assumptions can be removed by approximation arguments, since 
the final inequality we get does not depend on these bounds.  
Then 
we see that to check the Jensen property, we need only check functions of 
the form $(c/p)^{1/(p-1)} w(x)^{-1/(p-1)}$.  Now, in this case the Jensen 
condition is unchanged by multiplying $f$ by a constant, so we may 
simply check it for $f = w^{-1/(p-1)}$.  
Using the fact that $(w^{-1/(p-1)})^p w = w^{-1/(p-1)}$, 
the Jensen condition becomes
\[
\left( \fint w^{-1/(p-1)} \, dx \right)^p \leq \frac{C}{w(I)} \int w^{-1/(p-1)} \, dx
\]
which simplifies to 
\[
\left(\fint_I w \, dx \right) \left( \fint w^{-1/(p-1)} \right)^{p-1} \leq C
\]
which is the usual definition of the $A^p$ condition. 
\end{example}

\begin{example}
Let $\Phi_x(t) = w(x) t^{p(x)}$.  We will for now assume that 
$w$ is continuous and bounded above and below, although these 
assumptions can be removed by approximation arguments.  
We will also assume that $p$ is bounded below by a number greater 
than $1$, is bounded above, and is continuous.  Then 
we see that to check the Jensen property, we need only check functions of 
the form $(c/p(x))^{1/(p(x)-1)} w(x)^{-1/(p(x)-1)}$.

Notice that in this case, if the Jensen property is satisfied for a
particular $f$, it will be satisfied for any function of the form
$\alpha(x) f(x)$, where $\alpha$ is bounded above and below away from $0$.
The constant may be different, but depends only on the bounds of $\alpha$.
Thus, we need only check functions of the form
\[
  f(x) = c^{1/(p(x)-1)} w(x)^{-1/(p(x)-1)}.
\]
\end{example}

\section{Maximal Functions and the Jensen Property}
We begin with the following lemma.

\begin{lemma}\label{lemma:maximal}
Let $\Phi$ and $\Psi$ be families such that 
$\Psi$ has the $A_\infty$ property and the 
good ratio property. 
Suppose for each $z$ in the unit disc, 
there is a function 
$\Phi_z(t)$ such that there is some constant $K$ such that 
$\Phi_z(t) \geq K \Phi_{I_z}(t)$, where 
$I_z$ is the base of the tent with vertex $z$. 
 For each $x$ on the unit circle, define  
\[
\Phi(f)^*(x) = \sup_{z \in \Gamma(x)} \Phi_z(f(z))
\]  
where $\Gamma(x)$ is the cone with vertex $x$. 

Then
for every $A > 0$, there is a $C$ such that if 
\[
\rho_{\Psi}(f^*) = A
\]
then 
\[
 \int \Psi_x ( \Phi_x^{-1} ((2/K) \Phi(f)^*(x)) \, dx \geq C.
\]

If instead we assume that $\Psi \circ \Phi$ has the $A_\infty$ property
and the 
good ratio property, and the family $\Psi$ is uniformly doubling, 
then
for every $A > 0$, there is a $C>0$ such that if 
\[
\rho_{\Psi\circ \Phi}(f^*) = A
\]
then
\[
\int \Psi_x (  ( \Phi(f)^*(x)) \, dx \geq C'.
\]
\end{lemma}

\begin{proof}
Let $B$ be a basic interval in the set 
$\{x : f(x) > \lambda\}$.  We can cover $B$ by intervals such that 
each is the base of the tent for a point $z$ in the unit circle 
such that $f(z) > \lambda$.   By the Vitali covering lemma, there is a 
disjoint countable set of such intervals whose total measure is at least 
$1/5$ of the measure of $B$.  Let $I_z$ be some such interval, where 
$z$ is the top of the tent for which $I_z$ is the base.  

Recall that $\Phi_{I_z}(\lambda)$ is the average of $\Phi_x(\lambda)$ on the 
interval $I_z$.  Thus $\Phi_x(\lambda) \leq 2 \Phi_{I_z}(\lambda)$
on a set of 
measure at least $| I |/2$. 
Thus, 
$\Phi_{z}(f(z)) \geq \frac{K}{2} \Phi_x (\lambda)$
on that set.
So 
\[
\Phi_x^{-1}((2/K) \phi_{I_z}(f(z))) \geq \lambda
\]
on that set.  
%

Now, for each $z$ in the unit disc, recall that 
\[
\Phi(f)^*(x) = \sup_{z \in \Gamma(x)} \Phi_{z}(f(z)).
\]  
Thus 
\[
m( \{  \Phi_x^{-1}( (2/K) \Phi(f)^*(x)) > \lambda \} \cap B) > (1/5) m(B)
\]
for any fundamental interval $B$ in $\{x : f^*(x) > \lambda\}$. 

Now suppose $\Psi$ is a family with corresponding measures 
$\mu_\lambda$.  Also suppose that $\Psi$ has the  $A_\infty$ property
and the good ratio property.  
Suppose that $\rho_{\Psi}(f^*) = A$.  
It then follows by Theorem 
\ref {thm:twofunctionmeasureinequality} 
 there is some constant $C$ such that 
\[
 \int \Psi_x ( \Phi_x^{-1} ((2/K) \Phi_x(f)^*(x)) \, dx \geq C.
\]

Now, suppose that $\Psi \circ \Phi$ has the good ratio property
and the $A_\infty$ property, and moreover that $\Psi$ is 
uniformly doubling.  We may replace 
$\Psi$ with $\Psi \circ \Phi$ in the above to see that if 
\[
\int \Psi_x(\Phi_x(f^*(x))) \, dx = A 
\]
then 
\[
\int \Psi_x (  (2/K) \Phi(f)^*(x)) \, dx \geq C
\]
Since $\Psi$ is uniformly doubling we have that 
\[
 \int \Psi_x (  ( \Phi(f)^*(x)) \, dx \geq C'
\]
for some constant $C'$ depending only on the doubling constant 
and $C$.  
\end{proof}

There is a stronger version of the above result if 
$\Phi$ is slowly changing on intervals.  
\begin{lemma}\label{lemma:maximalsci}
Suppose $\Phi$ is slowly changing on intervals
and uniformly doubling.  Then for every $A$ such that 
$\rho(f) = A$, there is a constant $C$ such that 
\[
\Phi_x( (M_{HL}f)(x)) \leq C \sup_{I \ni x} \Phi_I \left(\fint_I f(x) \, dx\right).
\]
\end{lemma}
\begin{proof}
Let $\delta > 0$.  
We can find an $\varepsilon > 0$ and an interval $I$ containing 
$x$ such that 
\[
  M_{HL}f(x) \leq (1+\varepsilon) \fint_{I_x} f(y) \, dy.
\]
and
such that 
\[
  \Phi_{x}(M_{HL} f(x)) \leq
  \Phi_{x}\left((1+\varepsilon)\fint_{I_x} f(y) \, dy\right)
  \leq
  (1+\delta) \Phi_{x}\left(\fint_{I_x} f(y) \, dy\right).
\]
By the condition of slowly changing on intervals, this is at most 
\[
C (1+\delta) \Phi_{I,-}\left(\fint_{I_x} f(y) \, dy\right).
\]
This is at most
\[
C (1+\delta) \Phi_{I}\left(\fint_{I_x} f(y) \, dy\right).
\]
Now let $\delta \rightarrow 0$.  
\end{proof}
Because of this theorem, we can prove an analogue of the latter part of
Lemma \ref{lemma:maximal} under the assumption that
$\Phi$ is \sci\ and $\Psi$ is uniformly doubling, and we do not need to
assume that $\Psi \circ \Phi$ has the good ratio or $A_\infty$ property.

Now suppose that in addition to the assumptions of 
Lemma \ref{lemma:maximal}, or the analogue to that Lemma that
follows from Lemma \ref{lemma:maximalsci}, 
$\Phi$ has the Jensen property, and 
also that 
$\Psi$ is $L^1$ dominating and uniformly doubling.
Suppose in addition that 
$f$ is defined on the unit circle and that 
$f(z) = \fint_{I_z} f(x) \, dx$ for each $z$ in the unit disc.
This means that $f^* = M_{HL}(f)$.  

Then $\rho_{\Phi}(f) \leq A'$ for the constant $A'$ in the 
definition of $L^1$ dominating.  
Then 
for each $z$ in the unit disc, we have that 
\[
\Phi_{I_z}(f(z)) \leq C'' \fint_{I_z} \Phi_x(f(x)).
\]
by the Jensen property.  
The constant $C''$ here depends on $\rho_{\Phi}(f)$, and 
thus on $A$.  
Therefore
\[
\Phi(f)^*(x) \leq C'' K M_{HL}(\Phi_x(f)).
\]
It then follows that 
\[
 \int \Psi_x (  C '' K M_{HL}(\Phi_x(f)) ) \, dx \geq C'.
\]
 From the fact that $\Psi$ is uniformly doubling, it follows that 
\[
 \int \Psi_x (  M_{HL}(\Phi_x(f)) ) \, dx \geq C'''
\]
for some constant $C'''$ that depends on $A$.  
So for every $A$ there is a $B = C'''$ such that if 
$\rho_{\Psi \circ \Phi}(f^*) = A$ then 
$\rho_{\Psi}(M_{HL}(\Phi_x(f))) \geq B$.  

Now suppose that for every $B$ there is a $B'$ such that if 
$\rho_{\Psi}(M_{HL} g) \geq B$ then $\rho_{\Psi}(g) \geq B'$.
It then follows that 
$\rho_{\Psi}(\Phi_x(f)(x)) \geq B'$, i.e.\ that 
$\rho_{\Psi \circ \Phi}(f) \geq B'$.  

Therefore, by Lemma \ref{lemma:geqtoleq}, we have the 
following theorem.
\begin{theorem} \label{thm:maximalbound}
Either suppose 
 that the family $\Psi \circ \Phi$ has the 
good ratio property and the 
$A_\infty$ property and that $\Phi$ has the Jensen property, 
or suppose that $\Phi$ is slowly changing on intervals and 
$\Phi_{I,-}$ is convex for each interval $I$.
Also assume that 
$\Psi$ is $L^1$ dominating, uniformly doubling, and that 
for every $B > 0$, there is a $B' > 0$ such that if
$\rho_{\Psi}(M_{HL} g) \geq B$ then $\rho_{\Psi}(g) \geq B'$.
Then
for every $A>0$ there is a $B$ such that 
\[
\| M_{HL} f \|_{\Psi \circ \Phi, (A)} \leq 
\| f \|_{\Psi \circ \Phi, (B)}(f).
\]
\end{theorem}
Notice that we do not need the normal assumption of 
finiteness of the left hand side of the inequality that 
comes from Lemma \ref{lemma:geqtoleq}.  To see this, we may 
apply the theorem to $f_n(x) = \min(f(x),n)$ and 
apply the montone convergence theorem as $n \rightarrow \infty$. 

We note the theorem also applies if $f$ is harmonic and we replace 
$M_{HL} f$ by a constant times the nontangential maximal function 
of $f$.

\section{Nontangential maximal functions and harmonic conjugates}



Let $f^*$ be the nontangential maximal function of $f$.
Notice that the
nontangential maximal function of the function $\log^{+} |f(z)|$ is
$\log^{+} f^*$.
%
If we let $u$ be the harmonic function that equals 
$\log^+|f|$ on the boundary, we can use the subharmonicity of 
$\log^+|f|$ and Theorem \ref{thm:maximalbound} to prove the following 
theorem.  
The idea to do this seems to go back to 
\cite{Stein-Weiss_harmonic_several} for the constant exponent case.

\begin{theorem}\label{thm:ntmaxbound}
Suppose that the assumptions of Theorem \ref{thm:maximalbound} hold, 
and that $\Psi \circ \Phi \circ \log^+$ is uniformly doubling.  
Then for every number $A > 0$ there is a $B > 0$ such that
\[
\| f^* \|_{\Psi \circ \Phi \circ \log^+, (A)} \leq 
\sup_{0 < r < 1} \| f_r \|_{\Psi \circ \Phi \circ \log^+, (B)}.
\]
\end{theorem}
The family $\Psi \circ \Phi \circ \log^{+}$ is formed by composing the 
functions $\Psi_x \circ \Phi_x \circ \log^{+}$. 
\begin{proof}
Let $0 < r < 1$.  
Apply Theorem \ref{thm:maximalbound} to the harmonic function $u$ with 
values on unit disc equal to $\log^+|f(re^{i\theta})|$, and use the 
fact that $\log^{+}|f(z)|$ is subharmonic, as well as the fact that 
$(\log^+ |f_r|)^* = \log^{+} f_r^*$,
where $f_r(z) = f(rz)$ is a dilation of $f$.  
Also use the fact that $M_{HL} u > K u^*$ for some $K$, and the fact that 
$\Psi \circ \Phi \circ \log^+$ is uniformly doubling.  
As 
$r \rightarrow 1$ the resulting maximal function of 
$\log^{+}|f(re^{i\theta})|$ increases to the 
maximal function of $\log^{+} f(e^{it})$.  
\end{proof}
Note that we work with dilations in the above proof, since the
subharmonicity property does not necessarily work if we
apply it to the whole unit disc at once.  For example,
the reciprocal of a singular inner function has
boundary values equal to $1$ in absolute value almost
everywhere, but such a function is always greater than or
equal to $1$ in absolute value inside the disc. 

This theorem has as a corollary the well known fact that
the nontangential maximal operator is bounded from
$H^p$ to $L^p$ for $0 < p < \infty$.  (Of course, it is
trivially bounded from $H^\infty$ to $L^\infty$).  To see this let 
$\Phi(t) = e^{qt}$ and $\Psi(t) = t^{p/q}$ for large enough $t$ and for 
$q < p$.  (For small $t < 1$ we adjust the definition of $\Phi$
so that it is continuous everywhere, linear for $t < 1$, and
$\Phi(0) = 0$.) 



\section{The Smirnov Class} 
Recall that a function analytic in the unit disc is in the
Nevanlinna class if
\[
  \frac{1}{2\pi}\int_0^{2\pi} \log^+|f(re^{i\theta})| \, d \theta
\]
is bounded independently of $r$.  This is equivalent to $f$ being the
ratio of an $H^\infty$ function with a nonvanishing $H^\infty$ function.
Any such function converges nontangentially almost everywhere on the
boundary.  
If the $H^\infty$ function in the denominator can be chosen to be
outer, we say $f$ is in the Smirnov class.  This happens if
and only if 
\[
  \lim_{r \rightarrow 1}
  \frac{1}{2\pi}\int_0^{2\pi} \log^+|f(re^{i\theta})| \, d \theta
  =
  \frac{1}{2\pi}\int_0^{2\pi} \log^+|f(e^{i\theta})| \, d \theta. 
\]

Suppose that the conditions of Theorem
\ref{thm:ntmaxbound} hold, and that the families 
$\Phi$ and $\Psi$ are as in the statement of the theorem. 
Let the modular family $\Theta$ be defined by 
  $\Theta = \Psi \circ \Phi \circ \log^{+}$.  Also suppose that 
  the family $\Psi \circ \Phi$ is $L^1$ dominating.  

Then any function in the space $H_{\Theta}$ is in the
Nevanlinna class, since $\Psi \circ \Phi$ is $L^1$ dominating.
Since the
maximal function $f^*$ satisfies the property that 
$\int_0^{2\pi} \log^+ f^*(e^{i\theta}) \, d\theta < \infty$ by 
Theorem \ref{thm:ntmaxbound}, 
the dominated convergence theorem implies that $H_{\Theta}$ is a
subset of the Smirnov class.  

Now suppose that $f$ is in the Smirnov class.  Then we may write
$f = BSF$ where $B$ is the Blaschke product of the zeros of $f$,
$S$ is a singular inner function, and $F$ is an outer function.
Furthermore,
\[
  F(z) = \exp \left\{ \frac{1}{2\pi}
    \int_0^{2\pi} \frac{e^{it} + z}{e^{it} - z} \log |f(e^{it})| \, dt,\
    \right\},
  \]
  and $\log |f(z)| \in L^1$.  
Let $F_b$ have the same definition as $F$, except with
$\log^+|f(e^{it})|$ instead of $\log|f(e^{it})|$ in the integral. 
Let $F_s$ have the same definition as $F$, except with
$-\log^-|f(e^{it})|$ instead of $\log|f(e^{it})|$ in the integral, 
where $\log^-(t) = \max(-\log(t), 0)$. 
Then $|F_s| \leq 1$.
Also, $|B| \leq 1$ and $|S| \leq 1$.

Now, suppose that $\log^+ |f(e^{it})| \in L_{\Theta}$.
Let $u$ be the harmonic function agreeing with
$\log^+ |f(e^{it})|$ on the boundary of the unit disc.
Since $|e^w| = e^{\Rp w}$ and the Poisson kernel is the real
part of $\frac{e^{it} + z}{e^{it} - z}$, we have
that
\[
  \log^{+} |F(z)| = \log |F_b(z)| = u(z)
\]
on the boundary of the unit disc. 
Since $F_b(z)$ is nonvanishing, 
$\log |F_b(z)|$ is harmonic, and so 
\[
 \log |F_b(z)| = u(z)
\]
even in the interior of the disc.  

For each $A > 0$, there is a $B > 0$ such that if 
$\rho_{\Psi\circ\Phi \circ \log^{+}}(f) \leq A$ then 
$\rho_{\Psi\circ\Phi}(u^*) \leq B'$ 
by Theorem \ref{thm:maximalbound}.  
But this implies that 
$\rho_{\Psi \circ \Phi}(u_r) \leq B'$ for all $0 < r \leq 1$,
which implies that $\rho_{\Theta}(f_r) \leq B$ for each $r$.
Thus any function $f \in H_{\Theta}$ is in the Smirnov class, and
any function in the Smirnov class with boundary values in
$L_{\Theta}$ is in $H_{\Theta}$. 




\section{Equivalence of Definitions of Hardy Space}
We summarize our results on Hardy spaces in the following theorem. 
\begin{theorem}\label{thm:hardydefs}
  Let $\family{\Phi}$ and $\family{\Psi}$ be modular families of functions. 
  Let $\family{\Theta}$ be the family defined by 
  $\Theta = \Psi \circ \Phi \circ \log^{+}$.  
  Suppose that:
  \begin{itemize}
  \item The family $\Theta$ is uniformly doubling and has the
    $A_\infty$ property and the good ratio property.
  \item The family $\Phi$ has the Jensen property.
  \item The family $\Psi$ is $L^1$ dominating, uniformly doubling, and
    the Hardy-Littlewood maximal operator is bounded for $\Psi$ in
    the sense that for all $A > 0$, there is a $B > 0$ such that if
    $\rho_{\Psi}(g) \leq B$ then
    $\rho_{\Psi}(M_{HL}g) \leq A$. 
  \item The family $\Psi \circ \Phi$ is $L^1$ dominating. 
  \end{itemize}
  The following are equivalent for an analytic function
  $f =u + iv$ with domain equal to $\mathbb{D}$: 
  \begin{enumerate}
    \item $f \in H_{\Theta}$.
    \item\label{mainthmlist:second} $f^* \in L_{\Theta}$.
    \item $u^* \in L_{\Theta}$.
    \item $Au \in L_{\Theta}$. 
    \item\label{mainthmlist:smirnov} The function $f$ is in the Smirnov class and
      $f(e^{i\theta}) \in L_{\Theta}$.
    \end{enumerate}
    Furthermore, given any $A > 0$, there exists a
    $B > 0$ such that if the modular of one of the above
    functions in \eqref{mainthmlist:second}--\eqref{mainthmlist:smirnov} 
    or $\rho_{H_\Phi}(f)$ is bounded by $B$,
    then the modular of
    any other of them is bounded by $A$. 
\end{theorem}
We remark that the inverse of a singular inner function shows
that the condition that $f$ is in the Smirnov class in
\eqref{mainthmlist:smirnov} is 
necessary, and that merely assuming $f$ has nontangential
limit almost everywhere and $f(e^{i\theta}) \in L_{\Phi}$ is not
enough to make $f \in H_{\Phi}$. 

We now give some examples. 
\begin{example}
Let $w(x)$ be a weight such that the maximal operator is bounded 
on $L^a(w)$ for some $a > 1$.  In other words, let 
$w$ be an $A_\infty$ weight. Let 
$p > 0$.
Define $\Phi_x(t) = \Phi(t) = e^{t p / a}$ for $t \geq 1$,
and adjust the definition of
$\Phi(t)$ for $t \leq 1$ so that $\Phi(t)$ is linear for $0 \leq t \leq 1$,
continuous everywhere, and passes through the origin.
Let $\Psi_x(t) = t^a w(x)$ for
$t \geq 1$. 
Then  the theorem in this section applies.  In this case, 
$\Theta_x(t) = t^p w(x)$ for large enough $t$. 
Note that if we apply this method to 
an $A^p$ weight, defined by 
\[
\fint_I w \, dx \left( \fint w^{-1/(p-1)} \right)^{p-1} \leq C,
\]
we do not need to know that the maximal operator is bounded 
on $L^p(w)$, but only on $L^a(w)$ for some $a$.

Another way to get the result is to suppose that 
$w \in A_p$ for some $p > 1$.
Then $w^{1/2} \in A_p$.  
Let
$\alpha(t) = e^{t/2}$ for $t \geq 2$.  
We adjust the definition of
$\alpha$ so that it is $0$ at the origin and linear for
$0 \leq t \leq 2$.  Let $\Phi_x(t) = (\alpha(t))^p w(x)^{1/2}$,
so that 
$\Phi_x(t) = e^{t p/2} w(x)^{1/2}$ for $t \geq 2$. Let  
$\Psi_x(t) = t^{2}$.
Notice that $\Phi_x(t)$ has the Jensen property because it 
is of the form $(\alpha(t))^p w(x)^{1/2}$, where
$\alpha$ is convex and the functions 
$t^p w(x)^{1/2}$ have the Jensen property.  
Then $\Theta_x(t) = t^p w(x)$ for $t \geq 2$, and
$\Theta$ is uniformly doubling and
$A_\infty$.  Notice that with this method, we did not
need to know beforehand that the maximal operator is bounded
on any space except unweighted $L^2$. 
\end{example}

\begin{example}
Suppose that $w \in A_p$ for some $p > 1$.  
Then $w \in A_q$ for some $q < p$.  Say that 
$p = b q$.  Then $w^{1/b} \in A^q$.  
Let $\Phi_x(t) = t^q w(x)^{1/b}$ and 
$\Psi_x(t) = t^{b}$.  Then 
$\Theta_x(t) = (\log^{+} t)^{p} w(x)$ and the 
theorem applies.
\end{example}

%
%

\begin{example} Suppose $p(x)$ is log-H\"{o}lder continuous 
and that $0 < p_{-} < p_{+} < \infty$.  
Let
$\Psi(t) = t^{p_-/q}$ for some fixed $q$ such that
$0 < q < p_{-}$.
Define
\[
  \Phi_{x}(t) = \begin{cases}
    e^{p(x) q/p_{-}} t^{p(x)} &\text{ for $0 \leq t < p_{-}/q$}\\
    e^{t p(x) q/p_{-}} &\text{ for $p_{-}/q \leq t$}
  \end{cases}.
\]
and $\family{\Theta} = \family{\Psi} \circ \family{\Phi} \circ \log^{+}$, so
$\Theta_x(t) = t^{p(x)}$ for $t \geq e^{p_{-}/q}$.  
  
With these definitions, the conditions
of Theorem \ref{thm:hardydefs} are satisfied.
The conditions on $\Psi$ are clear, as is the fact that
$\Psi \circ \Phi$ is $L^1$ dominating.  For the condition on $\Phi$, 
notice that since $t^{p(x)}$ is \sci\ by Example \ref{ex:powersci},
so is
$\Phi_x(t)= (\alpha(t))^{p(x)}$ by
Lemma \ref{lemma:scicomp}, where $\alpha(t) = e^{t q/p_{-}}$ for
$t > p_{-}/q$ and $\alpha(t)$ is linear and passes through the origin
for $t \leq p_{-}/q$.
(In fact, $\alpha(t)$ is even convex.) 
By Example \ref{ex:powersci}, the family $\Theta$ is \sci. 
  

This defines the variable
  exponent Hardy space $H^{p(\cdot)}$.  
\end{example}

\begin{example} 
Let $1 < s < q < \infty$ be given and suppose that 
$p(x)$ is $1/s$ H\"{o}lder continuous.  Let 
\[
\Phi_x(t) = 
            \max(e^{t p(x)s/q}, (1+t)^s) \text{ if $t > 1$}\\
\]
and for $0 \leq t < 1$ define $\Phi_x(t)$ so it is $0$
at the origin,
linear and continuous.
Let $\Psi_x(t) = \Psi(t) = t^{q/s}$. 
Define $\Theta_x = \Psi \circ \Phi_x \circ \log^+$, so 
\[
\Theta_x(t) = 
            \max(t^{p(x)}, (1+\log^+ t)^q) \text{ if $t > e$}.\\
\]
  We call the resulting Hardy space $H_{\Theta} = H^{p(\cdot)}_{\log^q}$.

   We claim that Theorem \ref{thm:hardydefs} applies to this space. 
The conditions are routine to check except for the conditions about 
\sci (which imply the conditions stated in the theorem). 

We first claim that $\family{\Phi}$ is \sci. Note that it is enough to
show that $\family{\Phi} \circ \log^+$ is \sci\ by Lemma
\ref{lemma:scicomp}, since the function $\exp$ is increasing.  Note
that
\[
\Phi_{x}(\log^+ t) = 
            \max(t^{p(x)s/q}, (1+\log^+ t)^s) \text{ if $t > 1$}.\\
\]

This family is \sci\ by Example \ref{ex:logsci}. 
Thus, $\Phi$ is \sci.
To prove that $\family{\Theta}$ is \sci, we may apply
Lemma \ref{lemma:powersci} or Example \ref{ex:logsci} again. 
Thus, the conclusions of Theorem \ref{thm:hardydefs}
hold.
\end{example}


\begin{thebibliography}{10}

\bibitem{DCU-VariableHardy}
David Cruz-Uribe and Li-An~Daniel Wang, \emph{Variable {H}ardy spaces}, Indiana
  Univ. Math. J. \textbf{63} (2014), no.~2, 447--493. \MR{3233216}

\bibitem{DCU_VarLebesgueBook}
David~V. Cruz-Uribe and Alberto Fiorenza, \emph{Variable {L}ebesgue spaces},
  Applied and Numerical Harmonic Analysis, Birkh\"auser/Springer, Heidelberg,
  2013, Foundations and harmonic analysis. \MR{3026953}

\bibitem{Diening_maximal}
L.~Diening, \emph{Maximal function on generalized {L}ebesgue spaces
  {$L^{p(\cdot)}$}}, Math. Inequal. Appl. \textbf{7} (2004), no.~2, 245--253.
  \MR{2057643}

\bibitem{Fefferman-Gundy-Silverstein-Stein_good-lambda}
R.~Fefferman, R.~Gundy, M.~Silverstein, and E.~M. Stein, \emph{Inequalities for
  ratios of functionals of harmonic functions}, Proc. Nat. Acad. Sci. U.S.A.
  \textbf{79} (1982), no.~24, part 1, 7958--7960. \MR{687499}

\bibitem{MR3220151}
Mitsuo Izuki, Eiichi Nakai, and Yoshihiro Sawano, \emph{Hardy spaces with
  variable exponent}, Harmonic analysis and nonlinear partial differential
  equations, RIMS K\^{o}ky\^{u}roku Bessatsu, B42, Res. Inst. Math. Sci.
  (RIMS), Kyoto, 2013, pp.~109--136. \MR{3220151}

\bibitem{Kokilashvili-Paatashvili2006}
V.~Kokilashvili and V.~Paatashvili, \emph{On {H}ardy classes of analytic
  functions with a variable exponent}, Proc. A. Razmadze Math. Inst.
  \textbf{142} (2006), 134--137. \MR{2294576}

\bibitem{Kokilashvili-Paatashvili2008}
\bysame, \emph{On the convergence of sequences of functions in {H}ardy classes
  with a variable exponent}, Proc. A. Razmadze Math. Inst. \textbf{146} (2008),
  124--126. \MR{2464049}

\bibitem{Kokilashvili-Paatashvili2015}
\bysame, \emph{On variable exponent {H}ardy classes of analytic functions},
  Proc. A. Razmadze Math. Inst. \textbf{169} (2015), 93--103. \MR{3453827}

\bibitem{MR2899976}
Eiichi Nakai and Yoshihiro Sawano, \emph{Hardy spaces with variable exponents
  and generalized {C}ampanato spaces}, J. Funct. Anal. \textbf{262} (2012),
  no.~9, 3665--3748. \MR{2899976}

\bibitem{Stein_harmonicanalysis}
Elias~M. Stein, \emph{Harmonic analysis: real-variable methods, orthogonality,
  and oscillatory integrals}, Princeton Mathematical Series, vol.~43, Princeton
  University Press, Princeton, NJ, 1993, With the assistance of Timothy S.
  Murphy, Monographs in Harmonic Analysis, III. \MR{1232192}

\bibitem{Stein-Weiss_harmonic_several}
Elias~M. Stein and Guido Weiss, \emph{On the theory of harmonic functions of
  several variables. {I}. {T}he theory of {$H^{p}$}-spaces}, Acta Math.
  \textbf{103} (1960), 25--62. \MR{0121579}

\end{thebibliography}

\providecommand{\bysame}{\leavevmode\hbox to3em{\hrulefill}\thinspace}
\providecommand{\MR}{\relax\ifhmode\unskip\space\fi MR }
\providecommand{\MRhref}[2]{%
  \href{http://www.ams.org/mathscinet-getitem?mr=#1}{#2}
}
\providecommand{\href}[2]{#2}

\end{document}